\documentclass[11pt,letterpaper]{amsart}
\usepackage{amssymb,amsfonts,amsthm,array,epsfig,graphics,graphicx,tabularx,appendix}

\usepackage[usenames,dvipsnames]{color}

\def\cF{{\mathcal F}}

\newtheorem{theorem}{{Theorem}}[section]
\newtheorem{proposition}[theorem]{{Proposition}}

\newtheorem{lemma}[theorem]{{Lemma}}

\newtheorem{question}[theorem]{{Question}}

\theoremstyle{definition}

\theoremstyle{remark}

\title[No Euler class $0$ taut foliation on $W$]{Note on Euler class $0$ taut foliation on the Whitehead link exterior}
\author{Yao Fan, Zhentao Lai and Bin Yu}
\date{\today}

\begin{document}
\maketitle

\begin{abstract}
This paper studies the existence of co-orientable taut foliations on 3-manifolds, particularly focusing on the Whitehead link exterior. We demonstrate fundamental obstructions to the existence of such foliations with certain Euler class properties, contrasting with the more permissive case of the figure-eight knot exterior. Our analysis reveals that most lattice points in the dual unit Thurston ball of the Whitehead link exterior cannot be realized as Euler classes of co-orientable taut foliations, with only limited exceptional cases. The proofs combine techniques from foliation theory, including saddle fillings and index formulas, with topological obstructions derived from Dehn surgery. These results yield general methods for studying foliations on cusped hyperbolic 3-manifolds. These techniques applied to broader classes of manifolds beyond the specific examples considered here.
\end{abstract}

\section{Introduction}\label{s.Int}
Let $W$ be a compact orientable $3$-manifold such that $\partial W$ is either empty or consists of finitely many tori. Let $\mathcal{F}$ be a co-orientable taut foliation on $W$ that is transverse to $\partial W$, and let $T\mathcal{F}$ denote the tangent bundle of the foliation $\mathcal{F}$. According to Thurston \cite{Thu}, the dual Thurston norm $x^*(e(T\mathcal{F}))$ lies in the dual unit ball $B^*(W)$, meaning that $x^*(e(T\mathcal{F})) \leq 1$.

Thurston's \emph{Euler class-one conjecture} states that when $W$ is closed, for every integral class $a \in H^2(W;\mathbb{Z})$ with $x^*(a) = 1$, there exists a taut foliation $\mathcal{F}$ on $W$ whose Euler class equals $a$. In \cite{Ya}, Yazdi provided the first class of counterexamples to this conjecture. He constructed infinitely many hyperbolic  3-manifolds for which an even lattice point on the boundary of the dual unit ball does not correspond to any co-orientable taut foliation\footnote{Recently, Yi Liu proved that every oriented closed hyperbolic 3-manifold admits some finite cover that serves as a counterexample \cite{Liu}.}. Yazdi subsequently posed two natural questions:

\begin{question}\label{q.1}
Which points in the dual unit ball can be realized as the Euler class of some taut foliation on $W$?
\end{question}

\begin{question}\label{q.2}
Which $3$-manifolds with positive first Betti number admit a taut foliation with trivial Euler class?
\end{question}

Certainly, Question \ref{q.2} represents a special case of Question \ref{q.1}, but it carries particular significance: the existence of a taut foliation $\mathcal{F}$ on $W$ with $e(T\mathcal{F}) = 0$ implies the existence of a non-singular vector field $X$ on $W$ that is tangent to $\mathcal{F}$. Moreover, when $\partial W \neq \emptyset$, we additionally require that $X$ is transverse to $\partial W$.

As a concrete example addressing Question \ref{q.2}, Yazdi proposed investigating the case where $W$ is the Whitehead link exterior. The Whitehead link $L$ consists of two components $L_1$ and $L_2$ (see Figure \ref{f.S1}). Through an isotopy, one can observe that $L_1$ and $L_2$ are interchangeable. 

For the Whitehead link exterior $W$, its boundary $\partial W$ comprises two tori $T_1$ and $T_2$, corresponding to $L_1$ and $L_2$ respectively. We establish a coordinate system on each torus boundary using the standard meridian-longitude coordinates, where the meridian is assigned slope $\infty$ and the longitude is assigned slope $0$. 
Under this convention, any simple closed curve on a torus boundary can be assigned a rational slope $p/q \in \mathbb{Q} \cup \{\infty\}$, representing a curve that winds $p$ times around the meridian and $q$ times around the longitude.

\begin{figure}[htb]
    \centering
    \includegraphics[scale=0.4]{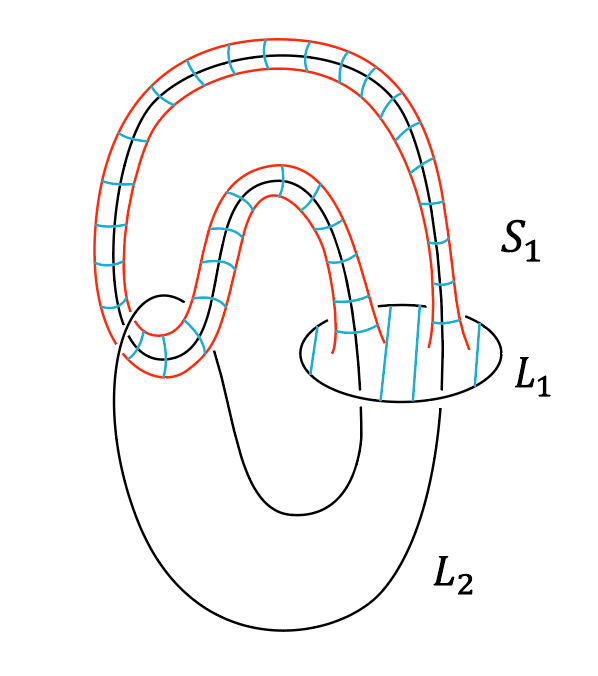}
    \caption{The Whitehead link $L$ and the surface $S_1$}
    \label{f.S1}
\end{figure}

Our results provide a direct answer to Question~\ref{q.2} for this specific case:

\begin{theorem}\label{t.WLc}
The Whitehead link exterior $W$ admits no co-orientable taut foliation $\mathcal{F}$ that is transverse to $\partial W$ and satisfies $e(T\mathcal{F}) = 0$ in $H^2(W, \partial W; \mathbb{Z})$.
\end{theorem}

This non-existence result for co-orientable taut foliations with vanishing Euler class on the Whitehead link exterior demonstrates that the origin of the dual unit ball $B_{x^*}(W)$ does not correspond to any such foliation.

According to Thurston \cite{Thu}, the dual unit ball $B_{x^*}(W)$ forms a square containing nine lattice points, as illustrated in Figure~\ref{TBWHL}. The $x$-axis represents the component $L_1$ of the Whitehead link $L$, while the $y$-axis corresponds to $L_2$. Regarding the boundary points of $B_{x^*}(W)$, Gabai established that the four vertices correspond to co-orientable taut foliations \cite{Ga1}. For the intermediate points $\{(\pm 1,0),(0,\pm 1)\}$, we demonstrate the following:

\begin{theorem}\label{t.WLc2}
Let $W$ be the Whitehead link exterior. There exists no co-orientable taut foliation $\mathcal{F}$ transverse to $\partial W$ with Euler class $e(T\mathcal{F}) = (\pm 1, 0)$ or $(0, \pm 1)$ in $H^2(W, \partial W; \mathbb{Z})$, except possibly when:

\begin{itemize}
    \item The restriction $\mathcal{F}|_{T_1}$ (respectively $\mathcal{F}|_{T_2}$) forms a suspension foliation, and
    \item The restriction $\mathcal{F}|_{T_2}$ (respectively $\mathcal{F}|_{T_1}$) is non-suspension (containing Reeb annuli) with all Reeb annuli cores having slope $\infty$.
\end{itemize}
\end{theorem}

We present the proof of Theorem~\ref{t.WLc2} in Section~\ref{s.mp}.

This brings us close to a complete understanding of Question~\ref{q.1} for the Whitehead link exterior, as summarized below:

\begin{itemize}
    \item Gabai's work \cite{Ga1} establishes that the four vertices of the dual unit ball $B_{x^*}(W)$ correspond to co-orientable taut foliations, represented by solid round dots in Figure~\ref{TBWHL}.
    
    \item Theorem~\ref{t.WLc} shows that the origin has no corresponding co-orientable taut foliations, indicated by a hollow round dot in Figure~\ref{TBWHL}.
    
    \item Theorem~\ref{t.WLc2} demonstrates that the lattice points $\{(\pm 1,0),(0,\pm 1)\}$ generally lack corresponding co-orientable taut foliations, except for the potential special case described in the theorem. These points are marked by hollow square dots in Figure~\ref{TBWHL}.
\end{itemize}

\begin{figure}[htb]
    \centering
    \includegraphics[scale=0.5]{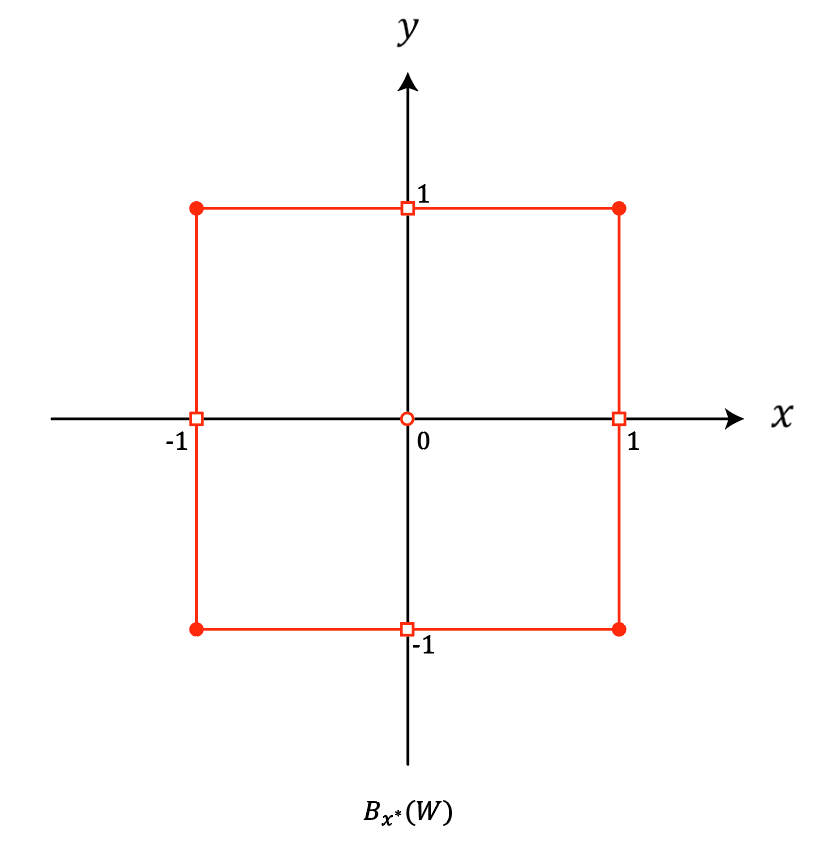}
    \caption{The four vertices of the square can be realized by taut foliations for the Whitehead link exterior $W$.}
    \label{TBWHL}
\end{figure}

This situation contrasts sharply with the case of the figure-eight knot exterior $N$, where every lattice point on the dual unit ball $B_{x^*}(N)$ corresponds to a taut foliation:

\begin{itemize}
    \item Following Thurston \cite{Thu}, one can readily verify that the dual unit ball $B_{x^*}(N)$ forms a line segment containing exactly three lattice points, as depicted in Figure~\ref{TBF8}.
    
    \item Gabai's results \cite{Ga1} establish that the two vertices of $B_{x^*}(N)$ correspond to co-orientable taut foliations. These foliations arise naturally from the fiber bundle structure of $N$, with each leaf being a once-punctured torus.
    
    \item For the origin, we construct the corresponding foliation as follows: Consider the natural suspension Anosov flow $X_t$ in the Sol manifold
    \[
    \bar{N} = \text{Map}\left(\mathbb{T}^2, \begin{pmatrix} 2 & 1 \\ 1 & 1 \end{pmatrix}\right).
    \]
    Let $\mathcal{F}^s$ denote the weak stable foliation of $X_t$, and let $\gamma_O$ be the periodic orbit corresponding to the origin $O \in \mathbb{T}^2$. By removing a small standard neighborhood of $\gamma_O$, we obtain an induced foliation $\mathcal{F}$ on $N$ derived from $\mathcal{F}^s$ on $\bar{N}$. This construction yields:
    \begin{itemize}
        \item $\mathcal{F}$ is transverse to $\partial N$
        \item $\mathcal{F} \cap \partial N$ consists of two Reeb annuli \cite{FW}
        \item There exists a vector field $X$ tangent to $\mathcal{F}$ and transverse to $\partial N$ \cite{FW}
    \end{itemize}
    Consequently, we have $e(T\mathcal{F}) = 0$.
\end{itemize}

\begin{figure}[htb]
    \centering
    \includegraphics[scale=0.5]{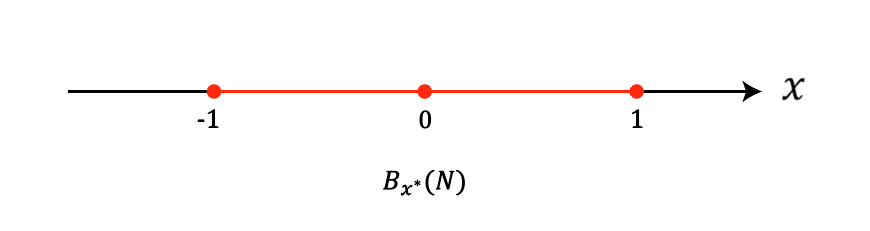}
    \caption{All three integer points can be realized by taut foliations for the figure-eight knot exterior $N$.}
    \label{TBF8}
\end{figure}

We outline the proof strategy for Theorem~\ref{t.WLc} at the beginning of Section~\ref{s.more}. The key observation is that when sufficient topological information is available for a hyperbolic 3-manifold with finitely many cusps, our method can be adapted to establish the non-existence of both foliations and hyperbolic plugs. 

Building upon the results of Dunfield, Hoffman, and Licata \cite{DHL}, we establish the existence of an infinite family $\{Y_k\}_{k\in\mathbb{Z}}$ (for $|k|\gg 0$) of one-cusped fibered hyperbolic 3-manifolds that admit no co-orientable taut foliations with vanishing Euler class that are transverse to the boundary. The complete technical details and further applications of this method are developed in Section~\ref{s.more}.

\section*{Acknowledgments}
The third author is supported by Shanghai Pilot Program for Basic Research, National Program for Support of Top-notch Young Professionals and the Fundamental Research Funds for the Central Universities.

\section{Saddle fillings}

We begin by describing a technique for constructing foliations on closed 3-manifolds. Let $M$ be a compact 3-manifold with torus boundary admitting a co-orientable taut foliation $\mathcal{F}$ that is transverse to $\partial M$. Suppose the induced foliation $f = \mathcal{F} \cap \partial M$ contains Reeb annuli. Then, following Gabai's saddle filling operation \cite[Operation 2.4.4]{Ga2}, we can extend $\mathcal{F}$ to a foliation $\overline{\mathcal{F}}$ on the Dehn filling of $M$, with possible exceptions in certain special cases. More precisely: 

\begin{theorem}[Gabai]\label{t.monkey}
Let $M$ be a compact orientable 3-manifold, $\cF$ be a co-orientable taut foliation on $M$ and $P$ be a torus boundary component such that $\cF$ transversely intersects $P$ and $\cF|_P$ has a Reeb component. If $\alpha$ is a simple closed curve on $P$ which is not isotopic to the core of a Reeb component of $\cF|_P$, then $M(\alpha)$ carries a taut foliation, where  $M(\alpha)$ is the manifold obtained by $\alpha$-filling $P$.
\end{theorem}

Gabai established the detailed proof for the specific case where the number of Reeb components is four and $\alpha$ intersects each annulus in exactly one arc. While the general case follows similarly for experts, our subsequent arguments will rely crucially on this result. To maintain the self-contained nature of our paper, we provide here a complete proof for the remaining cases, adapting Gabai's original approach. 

\begin{proof}[Proof of Theorem \ref{t.monkey}]
Suppose $M(\alpha)=M \cup V$ where $V$ is a product $D^2 \times S^1$ and $\alpha$ is the meridian of $V$, i.e. $\alpha = \partial D^2$. Since $\cF$ is co-orientable and  $\cF|_P$ has a Reeb component, then $\cF|_P$ has an even number of Reeb components. Assume the number of the Reeb components is $n$  and $\alpha$ intersects each annulus in $k$ essential arcs. Let $R_1, S_1,  \cdots R_n, S_n$ be the annuli cyclically ordered such that $\cF|_{R_i}$ is a Reeb component and $\cF|_{S_i}$ is a suspension of a homeomorphism $h_i$ of the interval (see Figure \ref{nk} as the illustration of the case $(n,k)=(4,3)$).

Let $c_i, c_i'$ be the boundary circles of $R_i$ ordered via the cyclic order. We can assign orientations $O_{\alpha}$, $O_{i}$, $O_{i}'$ to $\alpha$, $ c_{i}$, $c_{i}'$  respectively, such that the pairs $\{ O_{\alpha}, O_{i} \}$, $\{ O_{\alpha},O_{i}' \}$ are compatible with the outward normal orientation of $\partial V$.

\begin{figure}[htb]
    \centering
    \includegraphics[scale=0.35]{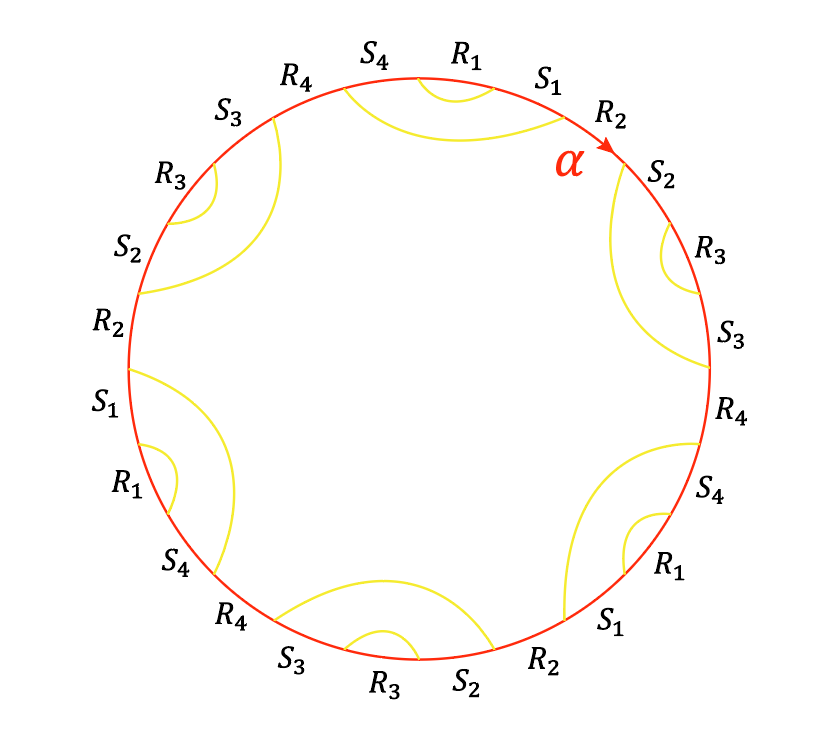}
    \caption{The case when $(n,k)=(4,3)$}
    \label{nk}
\end{figure}

The proof will consist of the following 4 steps.

\textbf{Step 1: Making $R_2, R_4,  \cdots R_n$ parallel.}

By the word ``parallel'', we mean the foliations in two Reeb annuli spiral in the same direction around the torus. Without loss of generality, suppose the foliations in each $R_{2i}$ spiral in the same direction as $O_{2i}$, $O_{2i}'$, otherwise turn $R_{2i}$ around as follows.

\begin{enumerate}
\item Connect $c_{2i}$ and $c_{2i}'$ with an annulus $A_{2i} \subset V$ where $A_{2i}$ and  $R_{2i}$ bound a solid torus $V_{2i} \subset V$. Then extend $\cF$ to $V_{2i}$ by attaching a half solid torus Reeb foliation as Figure \ref{halfR} shows. Note that the new foliation $\cF'$ is still taut, since each leaf of $\cF'$ is an extension of an old leaf from $\cF$.

\begin{figure}[htb]
    \centering
    \includegraphics[scale=0.4]{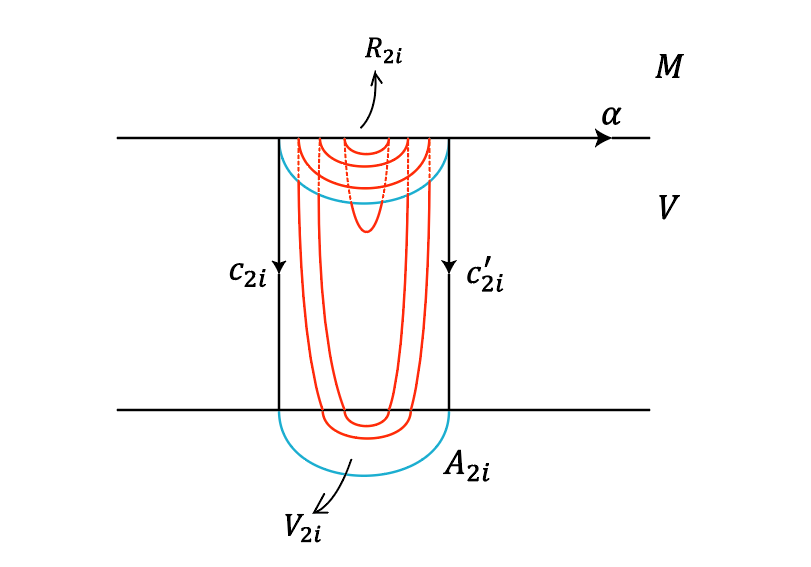}
    \caption{A half Reeb foliation}
    \label{halfR}
\end{figure}

\item Suppose $l_{2i}$ is the leaf of  $\cF'$ containing  $A_{2i}$. Split $l_{2i}$ to an $I$-bundle $U(l_{2i})$, where $\partial U(l_{2i})= l_{2i}^{+}  \cup  l_{2i}^{-}$  and  $c_{2i}$ (resp.,$\ c_{2i}'$) is split into $c_{2i}^{+}$, $c_{2i}^{-}$ (resp.,$\ c_{2i'}^{+}$, $c_{2i'}^{-}$). Take a product foliation in $U(l_{2i})$. For simplicity, we still denote the new foliation by $\cF'$.
\item Take a complete $1$-manifold $J$ in $l_{2i}^{+}$ that intersects $c_{2i}^{+}$, $c_{2i'}^{+}$ exactly once. By doing a ``Creating homology surgery'' (see \cite[Operation 2.2]{Ga2}), shown in Figure \ref{com} and Figure \ref{creating}, we get a new foliation $\cF^{''}$ from $\cF'$ such that the holonomy of $\cF^{''}$ along $c_{2i}^{+}$ is contracting.

\begin{figure}[htb]
    \centering
    \includegraphics[scale=0.4]{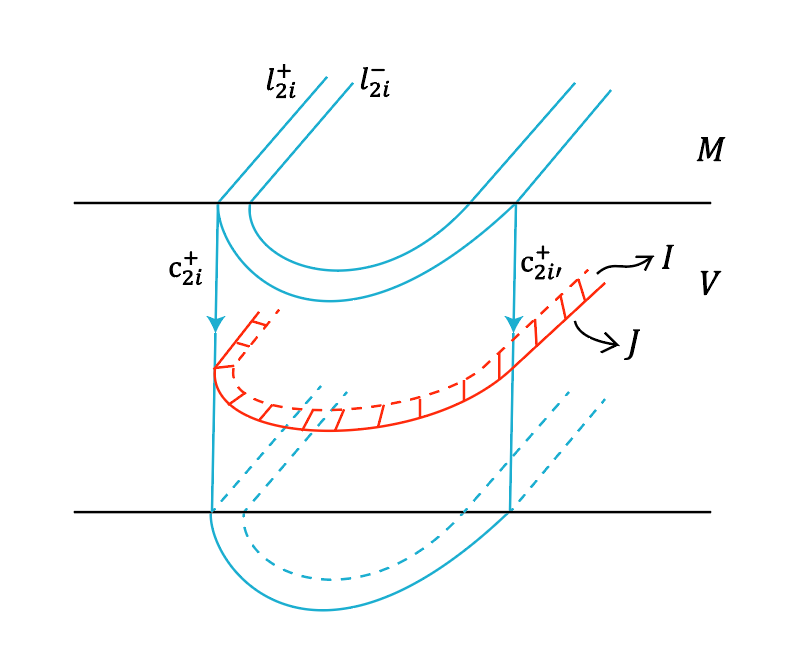}
    \caption{The complete $1$-manifold $J$}
    \label{com}
\end{figure}

\begin{figure}[htb]
    \centering
    \includegraphics[scale=0.4]{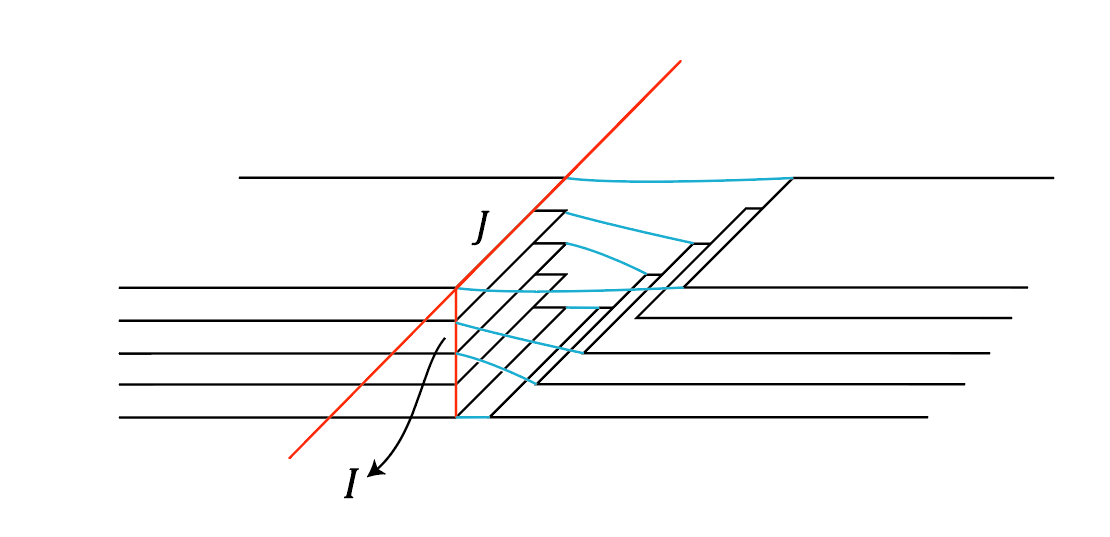}
    \caption{Creating homology surgery}
    \label{creating}
\end{figure}

\item Dig out an annulus $A_{2i}'$ in $U_{2i} \cap V$ such that $\partial A_{2i}' = c_{2i}^{+} \cup c_{2i'}^{+}$ and $\cF^{''}|_{A_{2i}'}$ is a Reeb annulus making the holonomy along $c_{2i}^{+}$ contracting (see Figure \ref{CR}).

\begin{figure}[htb]
    \centering
    \includegraphics[scale=0.4]{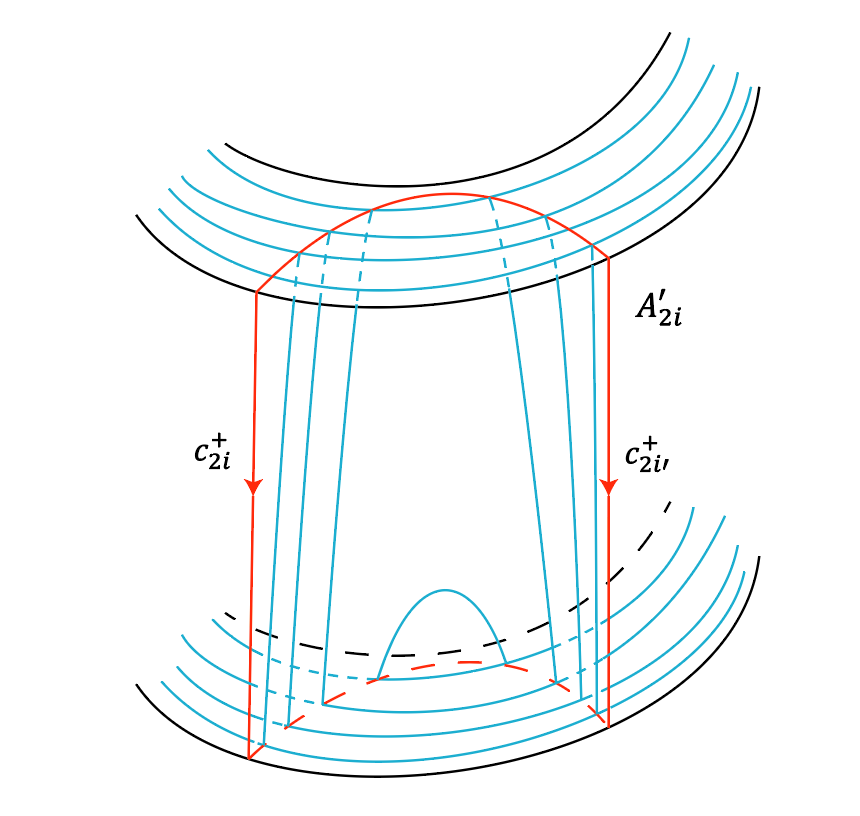}
    \caption{Creating a Reeb annulus}
    \label{CR}
\end{figure}

\end{enumerate}
For simplicity, we still denote by $R_{2i}$ the new Reeb annulus.

\textbf{Step 2: Adjusting the holonomy $h_{2i}$ of $S_{2i}$ and  $h_{2i+1}$ of $S_{2i+1}$ $ (i=1,2,\cdots, \frac{n}{2})$  such that they are conjugate  (Consider $h_{n+1}$ as $h_1$).}

Here $h_{2i}$ is a homeomorphism on an interval along the orientation of $\alpha$, while  $h_{2i+1}$ is along the reverse orientation of $\alpha$.

Gabai elaborated on this process in detail (see \cite[Operation 2.4.4]{Ga2}). The essential reason for this step is that $S_{2i}$ and $S_{2i+1}$ exactly bound one Reeb annulus $R_{2i+1}$. 

\textbf{Step 3: Capping off $R_1,R_3,\cdots, R_{n-1}$ with half solid torus Reeb foliations.  Extending $S_{2i}$ across the solid torus to $S_{2i+1} (i=1,2,\cdots, \frac{n}{2})$ by connecting them(Consider $S_{n+1}$ as $S_1$).}

These two steps are illustrated in Figure \ref{nk}.

\textbf{Step 4: Attaching cut off $\frac{nk}{2}$-saddles to the rest of the solid torus, where $\frac{nk}{2}$ is the number of times that $\alpha$ crosses the Reeb components(see Figure \ref{saddle} as an illustration of a $6$-saddle).}

\begin{figure}[htb]
    \centering
    \includegraphics[scale=0.4]{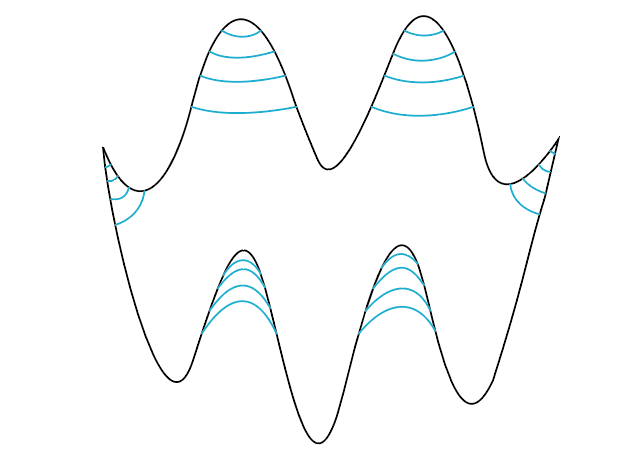}
    \caption{A $6$-saddle for $(n,k)=(4,3)$}
    \label{saddle}
\end{figure}

\end{proof}

\section{Proof of Theorem \ref{t.WLc}}\label{s.co-foli}

Consider the component $L_1$ of the Whitehead link. There exists a compact orientable surface $S_1$ with boundary $\partial S_1 = L_1$ that is disjoint from $L_2$, as illustrated in Figure~\ref{f.S1}.

Assume the Whitehead link exterior $W$ admits a co-orientable taut foliation $\mathcal{F}$ that is transverse to the boundary $\partial W = T_1 \cup T_2$ and has trivial Euler class $e(T\mathcal{F}) = 0$ in $H^2(W, \partial W; \mathbb{Z})$. 

According to Calegari \cite{Cal1}, any topological foliation on a 3-manifold is topologically isotopic to a $C^{\infty,0}$ foliation. This regularity result ensures the existence of a well-defined tangent bundle $T\mathcal{F}$ for the foliation. Moreover, the co-orientability condition guarantees that $T\mathcal{F}$ forms an orientable plane bundle over $W$.

Since $W$ is a compact 3-manifold with boundary, we consider the relative Euler class in the obstruction theory. Let $s$ be a section of the plane field $T\mathcal{F}$ restricted to the boundary $\partial W$. The relative Euler class $e(T\mathcal{F},s) \in H^2(W, \partial W; \mathbb{Z})$ vanishes precisely when $s$ extends to a global section of the plane bundle $T\mathcal{F}$ over $W$.

For our co-orientable taut foliation $\mathcal{F}$, we initially define $e(T\mathcal{F},s)$ using the outward-pointing section $s$. However, this definition is equivalent to using either the inward-pointing section or a section tangent to $\mathcal{F}|_{\partial W}$, since all such sections are related by homotopies through rotations in the plane field $T\mathcal{F}$. After fixing a particular section $s$, we will simply write $e(T\mathcal{F})$ for the relative Euler class.  

We assume that the foliation $f_i = \mathcal{F} \cap T_i$ contains no Reeb annuli. Under this assumption, the boundary of the surface $S_1$ must either be a leaf of $f_1$ or transverse to $f_1$. Following Yazdi's analysis \cite[Section 3.3]{Ya}, we obtain a generalization of Thurston's index formula:
\[
\text{Ind}(\mathcal{F},S_1) = \langle e(T\mathcal{F}), [S_1] \rangle.
\]
Moreover, the index satisfies the congruence relation
\[
\text{Ind}(\mathcal{F},S_1) \equiv \chi(S_1) \mod 2,
\]
where $\chi(S_1) = -1$ denotes the Euler characteristic of $S_1$. 

If there exists a co-orientable taut foliation $\mathcal{F}$ with $e(T\mathcal{F}) = 0$, then $\text{Ind}(\mathcal{F}, S_1) = 0$. However, this contradicts the congruence relation since $\chi(S_1) = -1 \not\equiv 0 \mod 2$. Therefore, the foliation $f_1$ must contain Reeb components on the torus boundary $T_1$. By the symmetry between $L_1$ and $L_2$, Reeb components must also appear on the other boundary component $T_2$.

This argument establishes the following key result:

\begin{proposition} \label{p.Reeb}
Let $W$ be the Whitehead link exterior. If there exists a co-orientable taut foliation $\mathcal{F}$ on $W$ that is transverse to $\partial W$ and satisfies $e(T\mathcal{F}) = 0$ in $H^2(W,\partial W; \mathbb{Z})$, then the induced $1$-dimensional foliation $f_i = \mathcal{F} \cap T_i$ contains a Reeb annulus for each boundary component $T_i$ ($i=1,2$).
\end{proposition}

Under these conditions, we can apply the saddle filling technique to complete the proof of Theorem~\ref{t.WLc}.

\begin{proof}[Proof of Theorem~\ref{t.WLc}]
Suppose, for contradiction, that there exists a co-orientable taut foliation $\mathcal{F}$ on $W$ transverse to $\partial W$ with $e(T\mathcal{F}) = 0$ in $H^2(W,\partial W;\mathbb{Z})$. Proposition~\ref{p.Reeb} guarantees that each induced 1-dimensional foliation $f_i = \mathcal{F}\cap T_i$ (for $i=1,2$) contains a Reeb annulus.

Applying the saddle filling technique, we perform Dehn fillings on $W$, using $\star$ to indicate no filling on a boundary component. The manifolds $W(\infty,\star)$ and $W(\star,\infty)$ are both trivial knot exteriors. For any non-zero rational number $r$, the filled manifolds $W(\infty,r)$ and $W(r,\infty)$ become lens spaces, while setting $r=0$ yields $S^3$ in both cases. However, neither lens spaces nor $S^3$ admit any taut foliations. 

By Theorem~\ref{t.monkey}, the only possible exceptional slope \footnote{Here an exceptional slope means the slope is the exceptional case in Theorem \ref{t.monkey}. Hence, it is parallel to the core of Reeb components.} on each boundary torus must be $\infty$. This conclusion, combined with the result of Calegari and Dunfield \cite{CD} that the Weeks manifold $W(5/1,5/2)$ admits no taut foliations, leads to a contradiction. Therefore, the Whitehead link exterior $W$ cannot support any co-orientable taut foliation $\mathcal{F}$ that is transverse to $\partial W$ and has vanishing relative Euler class.
\end{proof}

\section{Proof of Theorem \ref{t.WLc2}} \label{s.mp}

\begin{proof}[Proof of Theorem \ref{t.WLc2}]
By the symmetry of both the Whitehead link and its dual unit ball, the four lattice points $\{(\pm 1,0), (0,\pm 1)\}$ are equivalent. Without loss of generality, we consider the representative case $(1,0)$. 

Assume there exists a co-orientable taut foliation $\mathcal{F}$ on $W$ with Euler class $e(T\mathcal{F}) = (1,0)$ in $H^2(W,\partial W;\mathbb{Z})$ that is transverse to $\partial W$. From the definition of the relative Euler class, we derive the pairing equations:
\begin{align}
\langle e(T\mathcal{F},s), [S_1] \rangle &= 1, \label{e.1} \\
\langle e(T\mathcal{F},s), [S_2] \rangle &= 0, \label{e.2}
\end{align}
where $[S_1]$ and $[S_2]$ are generators of $H_2(W,\partial W)$ corresponding to the boundary tori $T_1$ and $T_2$ respectively.

Following the same argument as in Theorem~\ref{t.WLc}, equation (\ref{e.2}) implies that the induced foliation $f_2 = \mathcal{F}\cap T_2$ must contain Reeb annuli. If the foliation $f_1 = \mathcal{F}\cap T_1$ similarly contains Reeb annuli, then by the same reason in the proof of Theorem~\ref{t.WLc}, no such foliation $\mathcal{F}$ can exist.

The remaining case occurs when $f_1 = \mathcal{F} \cap T_1$ forms a suspension foliation while $f_2 = \mathcal{F} \cap T_2$ contains Reeb annuli. When the core curves of the Reeb annuli on $T_2$ have slopes different from $\infty$, we can perform $\infty$-slope saddle fillings at $T_2$. This operation produces a solid torus $W(\star,\infty)$ endowed with an induced taut foliation $\overline{\mathcal{F}}$. 

By construction, this foliation $\overline{\mathcal{F}}$ has two crucial properties: it intersects the boundary $\partial W(\star,\infty) = T_1$ transversely, and the induced 1-dimensional foliation $\overline{f}_1 = \overline{\mathcal{F}} \cap T_1$ retains its suspension structure.

To complete the proof of Theorem~\ref{t.WLc2}, we require the following fundamental result from the foliation theory. While likely known to experts, we provide a complete proof due to the absence of explicit references in the literature.

\begin{lemma}\label{l.storus}
Let $V$ be a solid torus. There exists no nontrivial taut foliation on $V$ that is transverse to the boundary, where "nontrivial" means the foliation is not a disk fibered over $S^1$.
\end{lemma}

\begin{proof}[Proof of Lemma~\ref{l.storus}]
Consider a solid torus $V$ equipped with a taut foliation $\mathcal{F}$ transverse to $\partial V = T$. Take two identical copies of $(V,\mathcal{F})$ and glue them along their boundaries via the identity map to construct a closed 3-manifold $M \cong S^2 \times S^1$ with an induced foliation $\widehat{\mathcal{F}}$.

Since $\mathcal{F}$ is taut, the doubled foliation $\widehat{\mathcal{F}}$ remains taut on $M$. However, Calegari's classification \cite[Theorem 4.35]{Cal2} shows that any taut foliation on $S^2 \times S^1$ must be trivial. As $(M,\widehat{\mathcal{F}})$ is the double of $(V,\mathcal{F})$, this forces $\mathcal{F}$ itself to be trivial, consisting entirely of disk leaves.
\end{proof}

Applying Lemma~\ref{l.storus} to our construction, we conclude that the foliation $\overline{\mathcal{F}}$ on the solid torus $W(\star,\infty)$ must be trivial. This immediately implies the absence of Reeb annuli in $T_2 \cap \overline{\mathcal{F}}$, leading to a contradiction with our earlier assumption. Consequently, no co-orientable taut foliation exists under these conditions, except for the singular case where the core curve of a Reeb annulus on $T_2$ has slope $\infty$. This completes the proof of Theorem~\ref{t.WLc2}. 
\end{proof}

\section{More examples}\label{s.more}
  Based on our analysis of the Whitehead link example, we can summarize a general method for obstructing the existence of co-orientable taut foliations with certain properties on 3-manifolds with boundary. Let $M$ be an orientable 3-manifold with nonempty boundary consisting of $n$ tori. The obstruction method proceeds as follows:

    \begin{description}
    \item[Step 1 (Index Analysis)] 
    
    For each boundary component $T_i$ ($i = 1,\dots,n$), construct an embedded orientable surface $S_i$ whose boundary $\partial S_i$ is an essential simple closed curve in $T_i$. Applying the index formula, we deduce that any co-orientable taut foliation $\mathcal{F}$ on $M$ transverse to $\partial M$ with vanishing relative Euler class must induce Reeb annuli on every boundary torus.

    \item[Step 2 (Dehn Filling Argument)]
    
    Employ the saddle filling technique on $M$. If sufficiently many Dehn fillings of $M$ yield 3-manifolds that admit no co-orientable taut foliations, we obtain a contradiction. This contradiction implies the non-existence of such a foliation $\mathcal{F}$ on the original manifold $M$.
    \end{description}

We now apply this two-step method to construct additional examples. Building on the work of Dunfield, Hoffman, and Licata \cite{DHL}, we consider a family of manifolds that is suitable for our approach. These 3-manifolds are obtained from the exterior of a specific two-component link $L$ in $S^3$, as illustrated in Figure~\ref{f.10}.
\begin{figure}[!h]
    \centering
    \includegraphics[width=3 in]{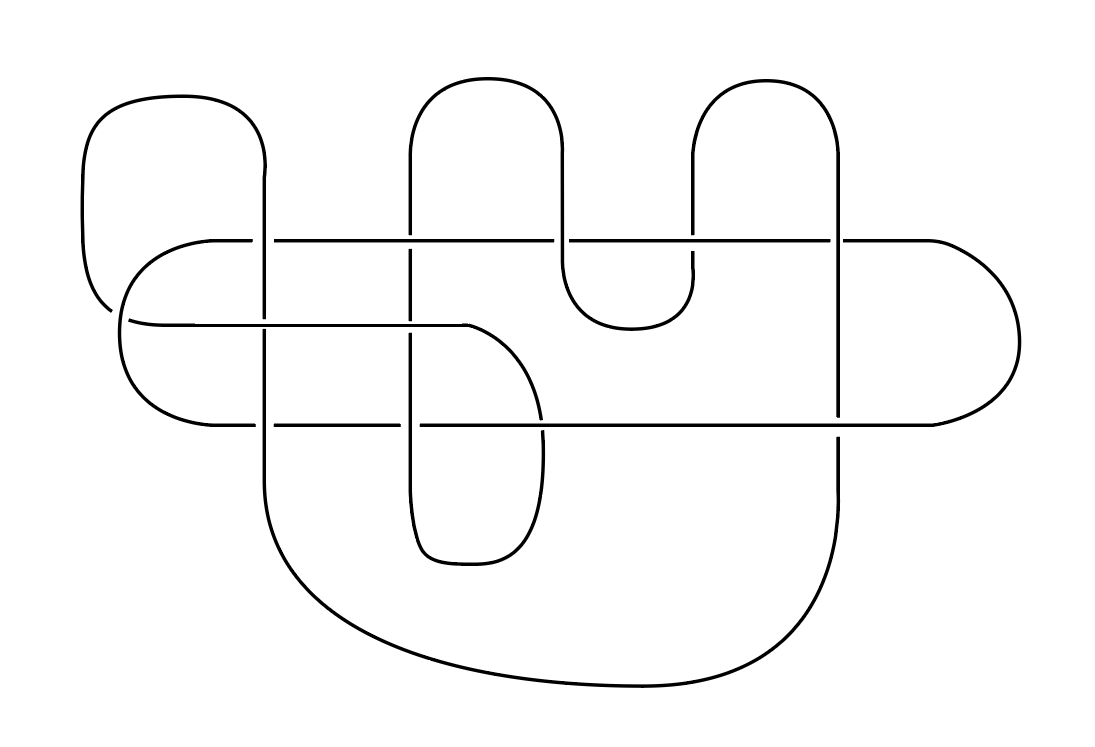}
    \caption{Link $L$ with 2 components}
    \label{f.10}
\end{figure}

Let $Y$ be the exterior of the link $L$. The symmetry between the two components of $L$ allows us to perform $(6k+1,k)$ Dehn filling on either component, yielding the 1-cusped hyperbolic 3-manifold $Y_k = Y(\frac{6k+1}{k},\star)$ (or equivalently $Y(\star,\frac{6k+1}{k})$). According to Theorem 4.4 of \cite{DHL}, each $Y_k$ satisfies three key properties: it admits two lens space Dehn fillings of coprime orders, represents the exterior of a knot in an integer homology sphere, and fibers over the circle with fiber a once-punctured surface.

The surface bundle structure guarantees the existence of a once-punctured surface $S_k \subset Y_k$ with $\partial S_k \subset \partial Y_k$. Following Step 1 of our general method, assume there exists a co-orientable taut foliation $\mathcal{F}$ on $Y_k$ transverse to $\partial Y_k$ with vanishing relative Euler class $e(T\mathcal{F}) = 0$. The same argument applied to the Whitehead link exterior shows that the induced 1-dimensional foliation $\mathcal{F} \cap \partial Y_k$ must contain Reeb annuli.

Proceeding to Step 2, we apply the saddle filling construction to $Y_k$. Lemma 4.6 of \cite{DHL} establishes that both $Y_k(\infty)$ and $Y_k(4)$ are lens spaces of coprime orders, neither of which can support a co-orientable taut foliation. Since there exists only one exceptional slope for $\partial Y_k$ when $k$ is fixed, this leads to a contradiction. We conclude that $Y_k$ cannot admit any co-orientable taut foliation with the specified properties.

\vskip 1cm
\noindent Yao Fan

\noindent {\small School of Mathematical Sciences}

\noindent {\small Key Laboratory of Intelligent Computing and
Applications (Tongji University), Ministry of Education}

\noindent{\small Tongji University, Shanghai 200092, CHINA}

\noindent{\footnotesize{E-mail: fanyao@tongji.edu.cn }}\\

\noindent Zhentao Lai

\noindent {\small School of Mathematical Sciences}

\noindent{\small Tongji University, Shanghai 200092, CHINA}

\noindent{\footnotesize{E-mail: zhentao\underline{ }lai@126.com }}\\

\noindent Bin Yu

\noindent {\small School of Mathematical Sciences}

\noindent {\small Key Laboratory of Intelligent Computing and
Applications (Tongji University), Ministry of Education}

\noindent{\small Tongji University, Shanghai 200092, CHINA}

\noindent{\footnotesize{E-mail: binyu1980@gmail.com }}

\vskip 2mm

\newpage

\appendix

\end{document}